\documentclass[11pt]{article}
\usepackage{graphicx}
\usepackage{amsthm, amsmath, amssymb, tikz, bm}
\usepackage{mathtools}
\usepackage{xifthen}
\usepackage{caption}
\usepackage[colorlinks=true,
linkcolor=blue,citecolor=blue,
urlcolor=blue]{hyperref}

\usepackage[margin=0.8in]{geometry} 

\usepackage[shortlabels]{enumitem}
\usepackage{todonotes}
\usetikzlibrary{calc,shapes, backgrounds}
\allowdisplaybreaks

\newtheorem{lemma}{Lemma}
\newtheorem{theorem}{Theorem}
\newtheorem{corollary}{Corollary}

\newtheorem{claim}{Claim}

\newtheorem{conjecture}{Conjecture}

\newcommand{\dss}{\displaystyle\sum}

\newcommand{\lp}{\left(}
\newcommand{\rp}{\right)}

\newcommand{\cx}{{\bf x}}
\newcommand{\cz}{{\bf z}}
\newcommand{\vol}{{\rm Vol}}

\DeclarePairedDelimiter{\ceil}{\lceil}{\rceil}
\DeclarePairedDelimiter{\floor}{\lfloor}{\rfloor}

\newcommand{\bx}{\mathbf{x}}

\newcommand{\one}{\mathbf{1}}

\title{On the maximum spread of planar and outerplanar graphs}
\author{
Zelong Li, \thanks{University of California, Los Angeles, Los Angeles, CA. ({\tt lizelong831@g.ucla.edu}).
 The author is partially supported by NSF DMS 2038080 grant through a summer REU program. }
\and 
William Linz, \thanks{University of South Carolina, Columbia, SC. ({\tt wlinz@mailbox.sc.edu}). The author is partially supported by NSF DMS 2038080 grant.}
\and Linyuan Lu, \thanks{University of South Carolina, Columbia, SC. ({\tt lu@math.sc.edu}). The author is partially supported by NSF DMS 2038080 grant.}
\and Zhiyu Wang \thanks{Georgia Institute of Technology, Atlanta, GA.
({\tt zwang672@gatech.edu}).}
}

\begin{document}

\maketitle

\abstract{
The spread of a graph $G$ is the difference between the largest and smallest eigenvalue of the adjacency matrix of $G$. Gotshall, O'Brien and Tait conjectured that for sufficiently large $n$, the $n$-vertex outerplanar graph with maximum spread is the graph obtained by joining a vertex to a path on $n-1$ vertices. In this paper, we disprove this conjecture by showing that the extremal graph is the graph obtained by joining a vertex to a path on $\ceil{(2n-1)/3}$ vertices and $\floor{(n-2)/3}$ isolated vertices. For planar graphs, we show that the extremal $n$-vertex planar graph attaining the maximum spread is the graph obtained by joining two nonadjacent vertices to a path on $\ceil{(2n-2)/3}$ vertices and $\floor{(n-4)/3}$ isolated vertices.
}

\section{Introduction}

Given a square matrix $M$, the \textit{spread} of $M$, denoted by $S(M)$, is defined as $S(M):= \max_{i,j} |\lambda_i -\lambda_j|$, where the maximum is taken over all pairs of eigenvalues of $M$. In other words, $S(M)$ is the diameter of the spectrum of $M$.
Given a graph $G=(V,E)$, the \textit{spread} of $G$, denoted by $S(G)$, is defined as the spread of the adjacency matrix $A(G)$ of $G$. Let $\lambda_1(G) \geq \cdots \geq \lambda_n(G)$ be the eigenvalues of $A(G)$. Since $A(G)$ is a real symmetric matrix, we have that the $\lambda_i$s are all real numbers. Thus $S(G) = \lambda_1 -\lambda_n$.

The systematic study of the spread of graphs was initiated by Gregory, Hershkowitz, and Kirkland \cite{GHK2001}. One of the central focuses of this area is to find the maximum or minimum spread over a fixed family of graphs and characterize the extremal graphs. Problems of such extremal flavor have been investigated for trees~\cite{AP2015}, graphs with few cycles~\cite{FWG2008, PBA2009, Wu-Shu2010}, the family of all $n$-vertex graphs~\cite{Aouchiche2008, BRTU2021+, Riasanovsky2021,Stanic2015, Stevanovic2014, Urschel2021}, the family of bipartite graphs~\cite{BRTU2021+}, graphs with a given matching number~\cite{LZZ2007}, girth~\cite{WZS2013}, or size~\cite{Liu-Liu2009}, and very recently for the family of outerplanar graphs~\cite{GBT2022}. We note that the spreads of other matrices associated with a graph have also been extensively studied (see e.g. references in \cite{GBT2022, Cao-Vince93, Cvetkovic-Rowlinson90}).

A graph $G$ is {\em planar} if it can be embedded in the plane, \textit{i.e.}, if it can be drawn on the plane in such a way that edges intersect only at their endpoints. A graph is {\em outerplanar} if it can be embedded in the plane so that all vertices lie on the boundary of its outer face. By Wagner's Theorem, a graph is planar if and only if it does not contain $K_5$ nor $K_{3,3}$ as a minor. Analogously, a graph is outerplanar if and only if it does not contain $K_4$ nor $K_{2,3}$ as a minor.
In this paper, we call an outerplanar graph $G$ {\em linear} if it contains a vertex $u$ that is adjacent to all other vertices. Here we call $u$ the {\em center} vertex of $G$. Similarly, we call a planar graph $G$ {\em linear} if it contains two vertices $u$ and $w$ that are each adjacent to all vertices of $V(G)\backslash\{u,w\}$, and call $u,w$ the {\em center} vertices of $G$. 
A graph $F$ is called a \textit{linear forest} if $F$ is a disjoint union of paths. 

Given two graphs $G$ and $H$, the \textit{join} of $G$ and $H$, denoted by $G\vee H$, is the graph obtained from the disjoint union of $G$ and $H$ by connecting every vertex of $G$ with every vertex of $H$. Let $P_k$ denote the path on $k$ vertices. Given two graphs $G$ and $H$, let $G\cup H$ denote the disjoint union of $G$ and $H$. Given a graph $G$ and a positive integer $k$, we use $kG$ to denote the disjoint union of $k$ copies of $G$. Given $v \subseteq V(G)$, let $N_G(v)$ denote the set of neighbors of $v$ in $G$, and let $d_G(v)$ denote the degree of $v$ in $G$, i.e., $d_G(v) = |N(v)|$. Given $S\subseteq V(G)$, define $N_G(S)$ as $N_G(S) = \{N_G(v):v\in S\}$. We may ignore the subscript $G$ when there is no ambiguity. 

There has been extensive research on finding the maximum spectral radius (i.e. the largest eigenvalue) of planar and outerplanar (hyper)graphs and the corresponding extremal (hyper)graphs; see, for example,~\cite{Boots-Royle91, Cao-Vince93, Cvetkovic-Rowlinson90, Dvorak-Mohar2010,Ellingham-Zha00, Yuan-Shu1999, Lin-Ning19, RF-Schwenk78, Rowlinson90, You-Liu2012, Yuan88, Yuan95, Yuan98, Ellingham-Lu-Wang22}.
In \cite{GBT2022}, Gotshall, O'Brien and Tait studied the maximum spread of outerplanar graphs and narrowed down the structure of the extremal graph attaining the maximum spread.

\begin{theorem}\cite{GBT2022}
For sufficiently large $n$, any graph which maximizes the spread over the family of outerplanar graphs on $n$ vertices is of the form $K_1\vee F$, where $F$ is a linear forest with $\Omega(n)$ edges.
\end{theorem}

In the same paper, Gotshall, O'Brien and Tait \cite{GBT2022} asked whether or not $F$ should be a path on $n-1$ vertices, and conjectured in the affirmative.
\begin{conjecture}\label{conj:gotconj}\cite{GBT2022}
For $n$ sufficiently large, the outerplanar graph on $n$ vertices with the maximum spread is given by $K_1 \vee P_{n-1}$.
\end{conjecture}

In this paper, we disprove this conjecture and determine the unique outerplanar graph attaining the maximum spread for sufficiently large $n$. 
\begin{theorem}\label{main}
For $n$ sufficiently large, the outerplanar graph on $n$ vertices with the maximum spread is  
  $K_1 \vee \left( P_{\left\lceil \frac{2n-1}{3} \right\rceil} \cup \left\lfloor \frac{n-2}{3}\right\rfloor  P_1\right)$.
\end{theorem}

We also extend our investigation to planar graphs, and determine the unique extremal $n$-vertex planar graph attaining the maximum spread.

\begin{theorem}\label{mainplanar}
For $n$ sufficiently large, the planar graph on $n$ vertices with the maximum spread is  
$(K_1 \cup K_1) \vee \left( P_{\lceil \frac{2n - 2}{3} \rceil} \cup \lfloor \frac{n-4}{3} \rfloor P_1\right).$
\end{theorem}

The remainder of the paper is organized as follows. In Section \ref{sec:outerplanar}, we first reduce the problem for outerplanar graphs to a special family of linear outerplanar graphs with only one non-trivial path (after deleting the center vertex). 
Theorem \ref{main} is proved by calculating the spread of these linear outerplanar graphs up to the error term $O(n^{-3})$.
In Section \ref{sec:planar}, we first prove that the maximum-spread planar graphs $G$ must contain $K_{2,n-2}$ as a subgraph by carefully estimating the eigenvectors of $\lambda_n$ and $\lambda_1$. From this structural lemma, we prove that $G$ is one of three types of planar graphs: a double wheel, a linear planar graph of the first kind, or a linear planar graph of the second kind (defined later). For each family, we calculate the maximum spread with enough precision to distinguish the maximum-spread planar graph. In the appendix, we prove a lemma on the existence of the Laurent series of the solution of certain equations. We believe that our method could be useful for other problems on determining the maximum or minimum spectral parameters of some family of graphs.

\section{Maximum spread over all outerplanar graphs}\label{sec:outerplanar}

Let $G$ be a simple graph and $\lambda_1\geq \ldots \geq \lambda_n$ be the eigenvalues of the adjacency matrix of $G$. Here $\lambda_1$ is called the \emph{spectral radius} of $G$. Let $A(G)$ denote the adjacency matrix of $G$. Given a vector $\cx \in \mathbb{R}^n$, let $\cx'$ denotes its transpose, and for each $i\in [n]$, let $\cx_i$ denote the $i$-th coordinate of $\cx$. 
Using the Rayleigh quotient of symmetric matrices, we have the following equalities for $\lambda_1$ and $\lambda_n$:
    \begin{align}
        \label{Rayleigh1}
        \lambda_1 &= \max_{\substack{\cx \in \mathbb{R}^n\\\cx\neq 0}} \frac{\cx' A(G) \cx}{\cx'\cx} = \max_{\substack{\cx \in \mathbb{R}^n\\\cx\neq 0}} \frac{2\sum_{ij\in E(G)} \cx_i \cx_j}{\cx'\cx} \\
        \label{Rayleighn}
         \lambda_n &= \min_{\substack{\cx \in \mathbb{R}^n\\\cx\neq 0}}  \frac{\cx' A(G) \cx}{\cx'\cx} = \min_{\substack{\cx \in \mathbb{R}^n\\\cx\neq 0}} \frac{2\sum_{ij\in E(G)} \cx_i \cx_j}{\cx'\cx}
    \end{align}

Consider a linear outerplanar graph $G= K_1 \vee (P_{\ell_1} \cup P_{\ell_2} \cup \cdots P_{\ell_r})$. Let $u$ be the center vertex of $G$ and $v_1,v_2,\ldots, v_{\ell}$ be the vertices in order of a path component $P_{\ell}$ in $G-u$ where $\ell \in \{\ell_i: i\in [r]\}$. Let $\alpha$ be a normalized eigenvector corresponding to an eigenvalue $\lambda$ of the adjacency matrix of $G$ so that $\alpha(u)=1$. Let $A_\ell$ be the adjacency matrix of $P_\ell$, $x_i=\alpha(v_i)$ for $1\leq i \leq \ell$,
$\bx=(x_1,\ldots, x_\ell)'\in {\mathbb R}^{\ell}$, and $\one = (1, \ldots, 1)' \in \mathbb{R}^{\ell}$. Let $I$ denote the identity matrix. The following lemma computes the vector ${\bf x}$.

\begin{lemma}\label{l1}
    Let $G= K_1 \vee (P_{\ell_1} \cup P_{\ell_2} \cup \cdots \cup P_{\ell_r})$ be a linear outerplanar graph on $n$ vertices with center vertex $u$. Suppose $\lambda$ is an eigenvalue of $A(G)$ with $|\lambda|\geq 2$, and $\alpha$ is a normalized eigenvector of $A(G)$ corresponding to $\lambda$ such that $\alpha(u)=1$. 
    Let $P_{\ell}$ be one of the path components of $G - u$ and let ${\bf x}$ and $A_{\ell}$ be defined as above.
    Then
    \begin{equation} \label{eq:taylor_eq} 
        {\bf x} =\sum_{k=0}^\infty \lambda^{-(k+1)} A_\ell^k \one.
    \end{equation}
\end{lemma}
\begin{proof}
Each vertex $v_i$ is adjacent to $u$ and $\alpha(u)=1$. Hence when restricting the coordinates of $A(G)\alpha$ to $v_1,\ldots, v_{\ell}$, we have that
\begin{equation}\label{eq:bx}
    A_\ell\bx + \one = \lambda \bx.
\end{equation}
It then follows that
\begin{align}
    \bx &= (\lambda I-A_\ell)^{-1}\one  \nonumber\\
    &=\lambda^{-1}  (I-\lambda^{-1}A_\ell)^{-1}\one  \nonumber\\
    &= \lambda^{-1} \sum_{k=0}^\infty (\lambda^{-1}A_\ell)^{k} \one\nonumber\\
    &= \sum_{k=0}^\infty \lambda^{-(k+1)} A_\ell^k \one. 
    \end{align}
Here we use the assumption that $|\lambda|\geq 2 > \lambda_1(A_\ell)$ so that the infinite series converges.
\end{proof}

\begin{corollary}\label{c1}
  Let $G= K_1 \vee (P_{\ell_1} \cup P_{\ell_2} \cup \cdots P_{\ell_r})$ be a linear outerplanar graph on $n$ vertices with center vertex $u$. Let $\lambda_1\geq \ldots \geq \lambda_n$ be the eigenvalues of $A(G)$ and $\{\alpha_i\}_{i\in [n]}$ be a set of pairwise orthogonal eigenvectors of $A(G)$ such that $\alpha_i$ corresponds to $\lambda_i$ and $\alpha_i(u)=1$ for each $i\in [n]$. Then the following properties hold:
  \begin{enumerate}[(i)]
    \item If $\lambda_i \geq 2$, then all entries of $\alpha_i$ are positive.
    \item If $\lambda_i \leq -2$,  then all entries of $\alpha_i$ but the $u$-entry are negative.
    \item For $2\leq i\leq n-1$, we have $\lambda_i\in(-2,2)$.
\end{enumerate}
\end{corollary}

\begin{proof}
Let $P_{\ell} = v_1 \ldots v_{\ell}$ be one of the path components in $G - u$. 
When $\lambda_i \geq 2$, by Equation \eqref{eq:taylor_eq}, since each entry in the vector $A_{\ell}^k \one$ is non-negative, it is clear that $\alpha_i(v_j) >0$ for each $j\in [\ell]$. This holds for every path component of $G-u$. Moreover, $\alpha_i(u)=1$. Hence all entries of $\alpha_i$ are positive.

For (ii), note that for each $j\in [\ell]$, $(A_{\ell}^k \one)_j$ counts the number of walks starting from $v_j$ in $P_{\ell}$. Since the degree of $v_j$ is at most $2$ in $G-u$ and $|\lambda|\geq 2$, we have that \{$|\lambda|^{-(k+1)} (A_{\ell}^k \one)_j, k\geq 0\}$ is a monotone decreasing sequence. Hence, if $\ell \ge 2$, then for each $j \in [\ell]$, the alternating series $\sum_{k=0}^{\infty} \lambda^{-(k+1)} (A_{\ell}^k \one)_j$ satisfies 
$$\sum_{k=0}^{\infty} \lambda^{-(k+1)} (A_{\ell}^k \one)_j
< \frac{1}{\lambda} (A_{\ell}^0\one)_j + \frac{1}{\lambda^2} (A_{\ell}\one)_j\leq \frac{1}{\lambda}+ \frac{1}{\lambda^2}\cdot 2 \leq 0,$$
since $\lambda\leq -2$ and $(A_{\ell}^k\one)_{j} > 0$ for each choice of $k$ and $j$. If $\ell = 1$, then we have $\sum_{k=0}^{\infty}\lambda^{-(k+1)} (A_{\ell}^k \one)_j = \frac{1}{\lambda} (A_{\ell}^0)_1 < 0$. 

 To prove (iii), we will prove that $\lambda_2 < 2$ and $\lambda_{n-1} > -2$. If $\lambda_2 \ge 2$, then by (i) the corresponding normalized eigenvectors $\alpha_1$ and $\alpha_2$ both have all entries positive, contradicting the spectral theorem that $\alpha_1\perp \alpha_2$.  Similarly, if $\lambda_{n-1} \le -2$, then by (ii) the corresponding normalized eigenvectors $\alpha_{n-1}$ and $\alpha_{n}$ both have all negative entries except for the $u$-entry (which is $1$), contradicting $\alpha_{n-1}\perp \alpha_{n}$.
\end{proof}

One of our main ideas in determining the extremal structure  is to merge the paths in the linear forests until there is only one non-trivial path. Consider a linear outerplanar graph
$G= K_1 \vee (P_{\ell_1} \cup P_{\ell_2} \cup \cdots P_{\ell_r})$.
A merge operation replaces $G$
by $G'=K_1\vee (P_{\ell_1+\ell_2-1} \cup P_1\cup \cdots P_{\ell_r})$, \textit{i.e.}, it merges two non-trivial paths in $G-u$ into one. Note that a merge operation does not change the number of edges in $G$. In the following two lemmas, we show that a merge operation increases $\lambda_1$ of $G$ and decreases $\lambda_n$ of $G$, and thus increases the spread of $G$.
\begin{lemma}\label{l2}
    Let $G= K_1 \vee (P_{\ell_1} \cup P_{\ell_2} \cup \cdots \cup P_{\ell_r})$ be a linear outerplanar graph on $n\geq 5$ vertices with $\ell_1, \ell_2 \geq 2$, and let $G'$ be obtained from $G$ by applying a merge operation, i.e., $G'=K_1\vee (P_{\ell_1+\ell_2-1} \cup P_1\cup \cdots \cup P_{\ell_r})$.  Then $\lambda_1(G') > \lambda_1(G)$.
\end{lemma}
\begin{proof}
Let $\lambda_1=\lambda_1(G)$ be the largest eigenvalue of $G$ and $\alpha$ be an eigenvector
of $A(G)$ corresponding to $\lambda_1$. Assume that $\alpha$ is normalized so that $\alpha(u)=1$ at the center vertex $u$. Moreover, assume $x_1,\ldots, x_{\ell_1}$ are the entries of $\alpha$ at the vertices $v_1, \ldots v_{\ell_1}$ of $P_{\ell_1}$, and $y_1, \ldots, y_{\ell_2}$ are the entries of $\alpha$ at the vertices $v_1', \ldots, v_{\ell_2}'$ of $P_{\ell_2}$. By the Perron-Frobenius Theorem (or Corollary \ref{c1}), all $x_i$'s and $y_j$'s are positive. 
Without loss of generality, we can assume $x_1\geq y_1$. We can then obtain an isomorphic copy of $G'$ from $G$ by adding an edge
$v_1 v_2'$ and deleting an edge $v'_1v'_2$. It then follows that
\begin{align*}
    \lambda_1(G') &\geq \frac{\alpha'A_{G'}\alpha}{\|\alpha\|^2}\\
    &= \frac{\alpha'A_{G}\alpha+2x_1y_2-2y_1y_2}{\|\alpha\|^2}\\
    &\geq \frac{\alpha'A_{G}\alpha}{\|\alpha\|^2}\\
    &=\lambda_1(G).
\end{align*}
Note that the equality cannot hold since $\alpha$ is not an eigenvector of $A_{G'}$ by Lemma \ref{l1}.
Thus, the merge operation on linear outerplanar graphs strictly increases the largest eigenvalue $\lambda_1$.
\end{proof}

\begin{lemma} \label{l3}
    Let $G= K_1 \vee (P_{\ell_1} \cup P_{\ell_2} \cup \cdots \cup P_{\ell_r})$ be a linear outerplanar graph on $n\geq 8$ vertices with $\ell_1, \ell_2 \geq 2$, and let $G'$ be obtained from $G$ by applying a merge operation, i.e., $G'=K_1\vee (P_{\ell_1+\ell_2-1} \cup P_1\cup \cdots \cup P_{\ell_r})$.  Then $\lambda_n(G') < \lambda_n(G)$.
\end{lemma}
\begin{proof}
Consider a linear outerplanar graph $G= K_1 \vee (P_{\ell_1} \cup P_{\ell_2} \cup \cdots)$ with $\ell_1, \ell_2\geq 2$. Let $G'=K_1\vee (P_{\ell_1+\ell_2-1} \cup P_1\cup \cdots)$ be the graph after merging. Let $\lambda_n=\lambda_n(G)$ be the smallest eigenvalue of $G$ and $\alpha$ be an eigenvector
of $A_G$ corresponding to $\lambda_n$. Assume that $\alpha$ is normalized so that $\alpha(u)=1$ at the center vertex $u$. Moreover, assume $x_1,\ldots, x_{\ell_1}$ are the entries of $\alpha$ at the vertices $v_1, \ldots v_{\ell_1}$ of $P_{\ell_1}$, and $y_1, \ldots, y_{\ell_2}$ are the entries of $\alpha$ at the vertices $v_1', \ldots, v_{\ell_2}'$ of $P_{\ell_2}$.

We claim now that $\lambda_n(G) \leq -2$ for any $n\geq 8$. For $8\leq n\leq 10$, it could be easily verified by computer. Suppose now $n\geq 10$. Let $\beta \in \mathbb{R}^n$ be a column vector indexed by the vertices of $G$ such that $\beta(u) = 1$ and $\beta(v) = -\frac{1}{\sqrt{n-1}}$ for any $v\neq u$. It follows by Equation \eqref{Rayleighn} that
$$\lambda_n \leq \frac{2\sum_{ij\in E(G)} \beta(i) \beta(j)}{\beta' \beta} = 2\cdot \frac{(n-2)\lp\frac{1}{\sqrt{n-1}}\rp^2-(n-1)\frac{1}{\sqrt{n-1}}}{(n-1)\lp\frac{1}{\sqrt{n-1}}\rp^2+1} \leq -2,$$
which holds for $n\geq 10$.

Now since $\lambda_n\leq -2$, all $x_i$'s and $y_j$'s are negative by Corollary \ref{c1}. Without loss of generality, we can assume
$|x_1|\leq |y_1|$. (Here we don't assume any relation between $\ell_1$ and $\ell_2$.)
We can then obtain an isomorphic copy of $G'$ from $G$ by adding an edge
$v_1 v_2'$ and deleting an edge $v'_1v'_2$. It now follows that
\begin{align*}
    \lambda_n(G') &\leq \frac{\alpha'A_{G'}\alpha}{\|\alpha\|^2}\\
    &= \frac{\alpha'A_{G}\alpha+2x_1y_2-2y_1y_2}{\|\alpha\|^2}\\
    &\leq \frac{\alpha'A_{G}\alpha}{\|\alpha\|^2}\\
    &=\lambda(G).
\end{align*}
Here we use the assumption $|x_1|\leq |y_1|$ and $x_1, y_1, y_2$ are all negative.
Note that the equality cannot holds since $\alpha$ is not an eigenvector of $A_{G'}$ by Lemma \ref{l1}.
Thus, the merge operation on a linear outerplanar graph strictly decreases the smallest eigenvalue $\lambda_1$.
\end{proof}

Hence we have the following corollary.
\begin{corollary}\label{cor:merge_S}
  Let $G= K_1 \vee (P_{\ell_1} \cup P_{\ell_2} \cup \cdots \cup P_{\ell_r})$ be a linear outerplanar graph on $n\geq 8$ vertices with $\ell_1, \ell_2 \geq 2$, and let $G'$ be obtained from $G$ by applying a merge operation, i.e., $G'=K_1\vee (P_{\ell_1+\ell_2-1} \cup P_1\cup \cdots \cup P_{\ell_r})$  Then $S(G') > S(G)$. 
\end{corollary}

Repeatedly applying Lemma \ref{l2} and Lemma \ref{l3}, we conclude
that the maximum spread is reached at a linear outerplanar graph with only one non-trivial path (after deleting the center vertex). Now we are ready to prove Theorem \ref{main}.

\begin{proof}[Proof of Theorem \ref{main}]
Let $G$ be a graph attaining the maximum spread among all outerplanar graphs on $n$ vertices (for sufficiently large $n$) and let $u$ be its center vertex.
By Corollary \ref{cor:merge_S}, the merge operation strictly increases
the spread of a linear outerplanar graph. Hence, we can assume that there is at most one non-trivial path in $G-u$. Thus, $G = G_\ell=K_1 \vee \left(P_\ell \cup (n-1-\ell)K_1\right)$ for some $\ell \in [n-1]$. 
Let $\lambda_1\geq \ldots \geq \lambda_n$ be the eigenvalues of $A(G_{\ell})$.
Suppose $\lambda \in \{\lambda_1, \lambda_n\}$ and $\alpha$ is a normalized eigenvector of $A(G_{\ell})$ corresponding to $\lambda$ with $\alpha(u)=1$. Let $\cx=(x_1,\ldots, x_\ell)$ be the vector of $\alpha$ restricted to the vertices of $P_\ell$. By the eigen-equation, the entry of $\alpha$ at those vertices not on $P_\ell$ and not $u$ is equal to $\frac{1}{\lambda}$.
The eigen-equation at $u$ is given by
\begin{equation} \label{eigen_u}
    \lambda =  (n-1-\ell)\frac{1}{\lambda} + \sum_{i=1}^\ell x_i.
\end{equation}
Applying Lemma \ref{l1}, we get
\begin{align*}
 \sum_{i=1}^\ell x_i &= \one' \cdot \bx\\
 &=\one' \cdot \sum_{k=0}^\infty \lambda^{-(k+1)} A_\ell^k \one \\
 &=   \sum_{k=0}^\infty \lambda^{-(k+1)} \one' A_\ell^k \one. 
\end{align*}
Plugging it into Equation \eqref{eigen_u}, we have 
\begin{equation} \label{eigen_u2}
    \lambda =    (n-1-\ell)\frac{1}{\lambda} +
  \sum_{k=0}^\infty \lambda^{-(k+1)} \one' A_\ell^k \one.  
\end{equation}

Recall that $P_{\ell} = v_1 v_2 \ldots v_{\ell}$ is the only nontrivial path in the neighborhood of $u$. Given $v_j \in V(P_{\ell})$, let $w_k(v_j)$ denote the number of walks of length $k$ starting from $v_j$
in $P_{\ell}$. Observe that for each $j \in [\ell]$, $(A_{\ell}^k \one)_j = w_k(v_j)$ and $\one' A_{\ell}^k \one = \dss_{j\in [\ell]} w_k(v_j)$. When
$k= 1$, $\one' A_\ell^k \one = \dss_{j \in [\ell]} d(v_j) = 2(\ell-2)+ 2 = 2(\ell-1)$. When $k = 2$,
it is not hard to see that $w_2(v_1) = w_2(v_{\ell}) = 2$; $w_2(v_2) = w_2(v_{\ell-1}) = 3$; and $w_2(v_j) = 4$ for all $j \in [3,\ell-2]$. Hence 
$$\one' A_\ell^2 \one = 4(\ell-4) + 3 \cdot 2 + 2 \cdot 2 = 4\ell-6.$$
Similarly, when $k = 3$, it is not hard to see that $w_3(v_1) = w_3(v_{\ell}) = 3$; $w_3(v_2) = w_3(v_{\ell-1}) = 6$; $w_3(v_3) = w_3(v_{\ell-2}) = 7$; and $w_3(v_j) = 8$ for all $j \in [4,\ell-3]$. Hence 
$$\one' A_\ell^3 \one = 8(\ell-6) + 7 \cdot 2 + 6 \cdot 2 + 3 \cdot 2 = 8(\ell-2).$$
Similarly, when $k=4$, it is not hard to see that $w_4(v_1) = w_4(v_{\ell}) = 6$; $w_4(v_2) = w_4(v_{\ell-1}) = 10$; $w_4(v_3) = w_4(v_{\ell-2}) = 14$;
$w_4(v_4) = w_4(v_{\ell-3}) = 15$;
and $w_4(v_j) = 16$ for all $j \in [5,\ell-4]$. Hence 
$$\one' A_\ell^4 \one = 16(\ell-8) + 15 \cdot 2 +
14 \cdot 2 + 10 \cdot 2 + 6 \cdot 2 = 16\ell-38.$$
Similarly, when $k=5$, it is not hard to see that $w_5(v_1) = w_5(v_{\ell}) = 10$; $w_5(v_2) = w_5(v_{\ell-1}) = 20$; $w_5(v_3) = w_5(v_{\ell-2}) = 25$;
$w_5(v_4) = w_5(v_{\ell-3}) = 30$; $w_5(v_5) = w_5(v_{\ell-4}) = 31$;
and $w_5(v_j) = 32$ for all $j \in [6,\ell-5]$. Hence 
$$\one' A_\ell^5 \one = 32(\ell-10) + 31 \cdot 2 +
30 \cdot 2 + 25 \cdot 2 + 20 \cdot 2 + 10 \cdot 2  = 32\ell-88.$$

Similarly, for $k\geq 6$, we have that
\begin{align*}
    \dss_{k=6}^{\infty} \lambda^{-k} \one' A_\ell^k \one \leq \dss_{k=6}^{\infty} \frac{2^k \ell}{\lambda^k} = O\left(\frac{\ell}{\lambda^6}\right).
\end{align*}

Multiplying both sides of \eqref{eigen_u2} by $\lambda$ and simplifying it, we get
\begin{equation} \label{eigen_main}
    \lambda^2 =    (n-1) +
    \frac{2(\ell-1)}{\lambda}
    + \frac{4\ell-6}{\lambda^2}
+ \frac{8(\ell-2)}{\lambda^3} + \frac{16\ell-38 }{\lambda^4} 
+  \frac{32\ell-88}{\lambda^5} +  O\left(\frac{\ell}{\lambda^6}\right).  
\end{equation}
Equation \eqref{eigen_main} has two real roots, which determines $\lambda_1$ and $\lambda_n$. $\lambda_1$ is near $\sqrt{n-1}$ and
$\lambda_n$ is near $-\sqrt{n-1}$.
By Lemma~\ref{lem:laurent} in the appendix, $\lambda_1$ has a series expansion
\begin{equation} 
    \lambda_1=\sqrt{n-1}+ c_1+ \frac{c_2}{\sqrt{n-1}}+ \frac{c_3}{n-1}+
\frac{c_4}{(n-1)^{3/2}}+ \frac{c_5}{(n-1)^2} + \frac{c_6}{(n-1)^{5/2}}+ O\left(\frac{1}{(n-1)^{3}}\right). 
\end{equation}
Plugging it into Equation \eqref{eigen_main}, and comparing the terms,
using SageMath, we compute the values of all $c_i$s as follows:
\begin{align}
    c_1&= \frac{\ell-1}{n-1}, \label{eq:c1} \\
    c_2&=2\left(\frac{\ell-1}{n-1}\right) - \frac{3}{2} \left(\frac{\ell-1}{n-1}\right)^2,  \label{eq:c2} \\
    c_3&= 4\left(\frac{\ell-1}{n-1}\right) - 8 \left(\frac{\ell-1}{n-1}\right)^2 
    +4\left(\frac{\ell-1}{n-1}\right)^3,\label{eq:c3} \\
    c_4&= -1 + 8\left(\frac{\ell-1}{n-1}\right) - 30 \left(\frac{\ell-1}{n-1}\right)^2 
    +35\left(\frac{\ell-1}{n-1}\right)^3
    -\frac{105}{8} \left(\frac{\ell-1}{n-1}\right)^4,  \label{eq:c4} \\ 
    c_5&= -4+ 20\left(\frac{\ell-1}{n-1}\right) 
    - 96\left(\frac{\ell-1}{n-1}\right)^2 
   +192\left(\frac{\ell-1}{n-1}\right)^3
    -160 \left(\frac{\ell-1}{n-1}\right)^4 
     +48 \left(\frac{\ell-1}{n-1}\right)^5.  \label{eq:c5} \\ 
       c_6&= -11+ 62\left(\frac{\ell-1}{n-1}\right) 
    - \frac{595}{2}\left(\frac{\ell-1}{n-1}\right)^2    +840\left(\frac{\ell-1}{n-1}\right)^3
        -1155 \left(\frac{\ell-1}{n-1}\right)^4 \nonumber \\
     & \hspace*{1cm}+\frac{3003}{4} \left(\frac{\ell-1}{n-1}\right)^5
 -\frac{3003}{16} \left(\frac{\ell-1}{n-1}\right)^6.     
     \label{eq:c6}
\end{align}
By a similar calculation, we get the following series expansion of $\lambda_n$:
\begin{equation} 
    \lambda_n=-\sqrt{n-1}+ c_1- \frac{c_2}{\sqrt{n-1}}+ \frac{c_3}{n-1}-
\frac{c_4}{(n-1)^{3/2}} + \frac{c_5}{(n-1)^2}-
\frac{c_6}{(n-1)^{5/2}}+ O\left(\frac{1}{(n-1)^3}\right). 
\end{equation}
Here $c_1, c_2, c_3,c_4,c_5,c_6$ are the same quantities as in Equations \eqref{eq:c1}, \eqref{eq:c2}, \eqref{eq:c3}, \eqref{eq:c4}, \eqref{eq:c5}, and \eqref{eq:c6}.
Observe that all $c_i$s are bounded since they are polynomials of $\frac{\ell-1}{n-1}$ over the interval $[0,1]$.
Thus, the spread of $G_\ell$, written by $f(\ell)$, can be expressed as follows:
  $$ f(\ell)= \lambda_1-\lambda_n = 2\sqrt{n-1} + \frac{2c_2}{\sqrt{n-1}} + \frac{2c_4}{(n-1)^{3/2}}
   + \frac{2c_6}{(n-1)^{5/2}} + O\left(\frac{1}{(n-1)^3}\right).$$
Since
    $$c_2(\ell) =2\left(\frac{\ell-1}{n-1}\right) - \frac{3}{2} \left(\frac{\ell-1}{n-1}\right)^2
    = \frac{2}{3}- \frac{3}{2}\left(\frac{\ell-1}{n-1}-\frac{2}{3}\right)^2
    \leq \frac{2}{3},$$
the function $c_2$ reaches the maximum at $\ell_1=\frac{2}{3}(n-1)+1$.
Let $\ell_0=\left\lceil\frac{2n-1}{3}\right\rceil$ be the target argument maximum of $f(\ell)$.
We have 
$$f(\ell_0)= 2\sqrt{n-1} + \frac{4}{3\sqrt{n-1}}+O\left(\frac{1}{(n-1)^{3/2}}\right).$$
\begin{claim}
There is a constant $C$ such that all argument maximum point(s) of $f(\ell)$ must belong to the interval
$$(\ell_1-C\sqrt{n-1}, \ell_1+C\sqrt{n-1}).$$
\end{claim}
\begin{proof}
Otherwise, for any $\ell$ not in this interval, we have
$$c_2(\ell)\leq \frac{2}{3}-\frac{3C^2}{2(n-1)}.$$
This implies
\begin{align*}
    f(\ell) &\leq 2\sqrt{n-1} + 2\frac{\frac{2}{3}-\frac{3C^2}{2(n-1)}}{\sqrt{n-1}} +O\left(\frac{1}{(n-1)^{3/2}}\right)< f(\ell_0).
\end{align*}
Here we choose the constant $C$ big enough
such that $$ -\frac{3C^2}{(n-1)^{3/2}} + O\left(\frac{1}{(n-1)^{3/2}}\right)<0.$$ \end{proof}
From now on,  we assume $\ell \in (\ell_1-C\sqrt{n-1}, \ell_1+C\sqrt{n-1})$.
Let us compute $f(\ell+1)-f(\ell)$. We have
\begin{align*}
    c_2(\ell+1)-c_2(\ell) &= \frac{2}{n-1}-\frac{3}{2}\frac{(2\ell-1)}{(n-1)^2}=\frac{2(n-1)-\frac{3}{2}(2\ell-1)}{(n-1)^2},\\
    c_4(\ell+1)-c_4(\ell) &= \frac{8}{n-1}-30\frac{(2\ell-1)}{(n-1)^2}
    + 35\frac{(3\ell^2-3\ell+1)}{(n-1)^3} -\frac{105}{8}\frac{(4\ell^3-6\ell^2+4\ell-1)}{(n-1)^4},\\
    c_6(\ell+1)-c_6(\ell) &=O\left(\frac{1}{n-1}\right).
\end{align*}
Plugging $\ell=\ell_1\cdot \left(1+ O\left(\frac{1}{\sqrt{n-1}}\right)\right)$
into $c_4(\ell+1)-c_4(\ell)$, we have
\begin{align*} 
  c_4(\ell+1)-c_4(\ell) &=   \frac{1}{n-1}\left(
 8-30\cdot 2\cdot \frac{2}{3} + 35  \cdot 3\cdot \left(\frac{2}{3}\right)^2
 -\frac{105}{8} \cdot 4\cdot \left(\frac{2}{3}\right)^3+ O\left(\frac{1}{\sqrt{n-1}}\right)\right)\\
 &=-\frac{8}{9(n-1)} + O\left(\frac{1}{(n-1)^{3/2}}\right).
\end{align*}
Therefore,
we have
\begin{align}
    f(\ell+1)-f(\ell) &=2 \frac{c_2(\ell+1)-c_2(\ell)}{\sqrt{n-1}}
    + 2 \frac{c_4(\ell+1)-c_4(\ell)}{(n-1)^{3/2}} 
    + 2\frac{c_6(\ell+1)-c_6(\ell)}{(n-1)^{5/2}} 
    + O\left(\frac{1}{(n-1)^{3}}\right) \nonumber \\
    &=\frac{4(n-1)-3(2\ell-1)}{(n-1)^{5/2}}
-\frac{16}{9(n-1)^{5/2}} +  O\left(\frac{1}{(n-1)^{3}}\right) \nonumber
\\
&= \frac{4n-6\ell-\frac{25}{9}}{(n-1)^{5/2}}
+  O\left(\frac{1}{(n-1)^{3}}\right). \label{eq:diff}
\end{align}
When $\ell\geq \ell_0$, we have
$$4n-6\ell-\frac{25}{9}\leq 4n-6\ell_0-\frac{25}{9}\leq 4n -6\cdot \frac{2n-1}{3}-\frac{25}{9}=-\frac{7}{9}<0.$$
Plugging it into Equation \eqref{eq:diff}, we have

$$f(\ell+1)-f(\ell) \leq  \frac{-\frac{7}{9}}{(n-1)^{5/2}}
+  O\left(\frac{1}{(n-1)^{3}}\right) <0.$$
When $\ell\leq \ell_0-1$, we have
$$4n-6\ell-\frac{25}{9}\geq 4n-6(\ell_0-1)-\frac{25}{9}\geq 4n - 6 \cdot\left (\frac{2n-1}{3}-\frac{1}{3}\right) -\frac{25}{9}=\frac{11}{9}>0.
$$
Thus, $$f(\ell+1)- f(\ell)\geq  \frac{\frac{11}{9}}{(n-1)^{5/2}}
+  O\left(\frac{1}{(n-1)^{3}}\right) >0.$$
Therefore, $f(\ell)$ reaches the unique maximum at 
$\ell_0$. This completes the proof of the theorem.
\end{proof}

\section{Maximum spread over all planar graphs}\label{sec:planar}
Let $G$ be the planar graph on $n$ vertices with maximum spread, and $\lambda_1\geq \ldots \geq \lambda_n$ be the eigenvalues of the adjacency matrix of $G$. 

\subsection{Structure of maximum spread planar graphs}
As a first step, we want to show that $G$ must contain $K_{2,n-2}$ as subgraph. We recall the result of Tait and Tobin \cite{TT2017} on the maximum spectral radius of planar graphs.

\begin{theorem}\cite{TT2017}\label{thm:taittobin}
For $n$ sufficiently large, the planar graph on $n$ vertices with maximal spectral radius is $P_2 \vee P_{n-2}$.
\end{theorem}

We first give some upper and lower bounds on $\lambda_1(G)$ and $|\lambda_n(G)|$ when $n$ is sufficiently large. We use known expressions for the eigenvalues of a join of two regular graphs~\cite[pg.19]{BH2012}.

\begin{lemma}\cite{BH2012}\label{lem:joinreglemma}
Let $G$ and $H$ be regular graphs with degrees $k$ and $\ell$ respectively. Suppose that $|V(G)| = m$ and $|V(H)| = n$. Then, the characteristic polynomial of $G\vee H$ is $p_{G\vee H}(t) = ((t-k)(t-\ell)-mn)\frac{p_G(t)p_H(t)}{(t-k)(t-\ell)}$. In particular, if the eigenvalues of $G$ are $k = \lambda_1 \ge \ldots \ge \lambda_m$ and the eigenvalues of $H$ are $\ell = \mu_1 \ge \ldots \geq \mu_n$, then the eigenvalues of $G\vee H$ are $\{\lambda_i: 2\le i\le m\} \cup \{\mu_j: 2\le j\le n\} \cup \{x: (x-k)(x-\ell)-mn = 0\}$. 
\end{lemma}

We will apply Lemma~\ref{lem:joinreglemma} to the graphs $(2K_1)\vee C_{n-2}$ and $K_2\vee C_{n-2}$. 

\begin{lemma}\label{lem:planarlambdan}
\[\sqrt{2n-4} - \frac{3}{2}-O\left(\frac{1}{\sqrt{n}}\right) \le |\lambda_n| \le \lambda_1 \le \sqrt{2n-4} +\frac{3}{2}+O\left(\frac{1}{\sqrt{n}}\right)  .\]
\end{lemma}

\begin{proof}
By Lemma~\ref{lem:joinreglemma}, $\lambda_1(K_2\vee C_{n-2})$ is the largest root of $(x-1)(x-2)-2(n-2)=0$, which is $\frac32 + \frac12 \sqrt{8n-15}$. Now
by Theorem~\ref{thm:taittobin}, we have $\lambda_1 \le \lambda_1(K_2\vee P_{n-2})\le \lambda_1(K_2\vee C_{n-2}) = \frac32 + \frac12 \sqrt{8n-15}$.
Note that $\frac32 + \frac{1}{2}\sqrt{8n-15} = \frac32 +\sqrt{2n-4} +O\left(\frac{1}{\sqrt{n}}\right)$. Hence, $-\lambda_n \le \lambda_1 \leq \frac32+\sqrt{2n-4} + O\left(\frac{1}{\sqrt{n}}\right)$. Since $G$ is the planar graph with maximal spread and $\lambda_1(K_{2,n-2}) = -\lambda_n(K_{2,n-2}) = \sqrt{2(n-2)}$, we have that 
$$\lambda_1 - \lambda_n \ge S(K_{2, n-2}) =\sqrt{2(n-2)}-(-\sqrt{2(n-2)}) = 2\sqrt{2n-4}.$$ Hence, $-\lambda_n \ge \sqrt{2n-4} - \frac32 - O\left(\frac{1}{\sqrt{n}}\right)$. 
\end{proof}

For the rest of this section, let ${\bf x}$ and ${\bf z}$ be the eigenvectors of $A(G)$ corresponding to the eigenvalues $\lambda_1$ and $\lambda_n$ respectively. For convenience, let ${\bf x}$ and ${\bf z}$ be indexed by the vertices of $G$. By the Perron-Frobenius theorem, we may assume that all entries of ${\bf x}$ are positive. We also assume that $\cx$ and $\cz$ are normalized so that the maximum absolute values of the entries of $\cx$ and $\cz$ are equal to $1$, and so
there are vertices $u_0$ and $w_0$ with ${\bf x}_{u_0} = |\cz_{w_0}| = 1$. 

Let $V_+=\{v\colon {\bf z}_v> 0\}$, $V_0=\{v\colon {\bf z}_v= 0\}$,
and $V_-=\{v\colon {\bf z}_v < 0\}$. 
Since $\cz$ is a non-zero vector, at least one of $V_{+}$ and $V_{-}$ is non-empty. By considering the eigen-equations of $\lambda_n \sum_{v\in V_{+}} \cz_v$ or $\lambda_n \sum_{v\in V_{-}} \cz_v$, both $V_{+}$ and $V_{-}$ are non-empty.
For any vertex subset $S$, we define the \textit{volume} of $S$, denoted by $\vol(S)$, as 
$\vol(S)= \sum_{v\in S} |\cz_v|$. 
In the following lemmas, we use the bounds of $\lambda_n$ to deduce some information on $V_{+}$, $V_{-}$ and $V_0$.

\begin{lemma}\label{lem:V0}
$|V_0| \leq \lp\frac32+o(1)\rp\sqrt{2n-4}$. 
\end{lemma}
\begin{proof}
    For any $v\in V_+$,
\begin{equation}\label{eq:zv}
  |\lambda_n|\cz_v = -\lambda_n \cz_v =-\sum_{u\in N(v)} \cz_u \leq  -\sum_{u\in N(v)\cap V_-} \cz_u =  \sum_{u\in N(v)\cap V_-} |\cz_u|.
\end{equation}
Similarly for any $u\in V_-$, we have
\begin{equation}\label{eq:zu}
     |\lambda_n| |\cz_u| \leq  \sum_{v\in N(u)\cap V_+} |\cz_v|.
\end{equation}
Summing over all $v\in V_+$ of Equation \eqref{eq:zv} and multiplying by $|\lambda_n|$, we get
\begin{align} \label{eq:sumzv}
    |\lambda_n|^2 \sum_{v\in V_+}|\cz_v| &\leq \sum_{v\in V_+}  \sum_{u\in N(v)\cap V_-} |\lambda_n| |\cz_u| \nonumber\\
    &\leq \sum_{v\in V_+}  \sum_{u\in N(v)\cap V_-} \sum_{y\in N(u)\cap V_+} |\cz_y| \nonumber\\
    &= \sum_{y\in V_+} |\cz_y| \cdot \sum_{u\in N(y)\cap V_-}|N(u)\cap V_+| \nonumber\\
    &= \sum_{y\in V_+} |\cz_y|  |E(N(y)\cap V_-, V_+)| \nonumber\\
    &\leq \sum_{y\in V_+} |\cz_y| (2 (|N(y)\cap V_-|+ |V_+|)-4).
\end{align}
In the last step, we use the fact that a bipartite planar graph on $m$ vertices can have at most $2m-4$ edges.
We use the trivial bound 
\begin{equation}\label{eq:yV1}
2 (|N(y)\cap V_-|+ |V_+|)-4 \leq 2 (|V_-|+ |V_+|) -4 = 2(n-|V_0|) -4.
\end{equation}
It then follows from \eqref{eq:sumzv} and \eqref{eq:yV1} that
\begin{equation}\label{eq:eqv0}
     |\lambda_n|^2 \vol(V_+) \leq   (2(n-|V_0|) -4)\vol(V_+).
\end{equation}
Applying the lower bound of $|\lambda_n|\geq \sqrt{2n-4}-\frac{3}{2}-O(\frac{1}{\sqrt{n}})$, and simplifying \eqref{eq:eqv0}, we then obtain that
  $|V_0| \leq \lp\frac32+o(1)\rp\sqrt{2n-4}$. This completes the proof of the lemma.
\end{proof}

Since $|V_-|+|V_+| = n -|V_0| = n - O(\sqrt{n})$, without loss of generality, we can assume $|V_-|> \frac{n}{2} - O(\sqrt{n})>20 \sqrt{2n-4}$. The following lemma bounds the volume of $V_{+}$.

\begin{lemma}\label{lem:V+}
Let $V_+'=\{v\in V_+\colon |N(v)\cap V_-|\geq |V_-| - 5\sqrt{2n-4}\}$ and $V_+''=V_+\setminus V_+'$. 
Then we have
\begin{equation}\label{eq:V+2prime}
    \vol(V''_+)  \leq\left( \frac{3}{7}+  O\left(\frac{1}{\sqrt{n}}\right)\right) \vol(V'_+),
\end{equation} and thus 
$\vol(V_+) \leq 
     \lp \frac{10}{7}+  O\left(\frac{1}{\sqrt{n}}\right)\rp \vol(V'_+)
     \label{eq:V+}$.
\end{lemma}
\begin{proof}
By Inequality \eqref{eq:sumzv} of Lemma \ref{lem:V0}, we have
\begin{equation}\label{eq:sumzv2}
    |\lambda_n|^2 \sum_{v\in V_+}|\cz_v| \leq \sum_{y\in V_+} |\cz_y| (2 (|N(y)\cap V_-|+ |V_+|)-4).
\end{equation}  
For $y\in V_{+}'$, we use the trivial bound
\begin{equation}\label{eq:yV1B}
2 (|N(y)\cap V_-|+ |V_+|)-4 \leq 2 (|V_-|+ |V_+|) -4 = 2(n-|V_0|) -4.
\end{equation}
For $y\in V''_+$, we use a better bound
\begin{equation}\label{eq:yV2}
    2 (|N(y)\cap V_-|+ |V_+|)-4 \leq 2 ((|V_{-}|-5\sqrt{2n-4})+ |V_+|) -4 \leq 2n-4 -10\sqrt{2n-4}.
\end{equation}
Plugging Equations \eqref{eq:yV1B} and \eqref{eq:yV2} into Equation \eqref{eq:sumzv2}, we get
\begin{align}
    |\lambda_n|^2 \sum_{v\in V_+}|\cz_v| &\leq  \sum_{y\in V_+} |\cz_y| (2 (|N(y)\cap V_-|+ |V_+|)-4) \nonumber\\
    &\leq \sum_{y\in V'_+} |\cz_y| (2n-4) + \sum_{y\in V''_+} |\cz_y| (2n-4 -10\sqrt{2n-4}) \nonumber \\
    &= (2n-4)\sum_{y\in V_+}|\cz_y| - 10\sqrt{2n-4}\sum_{y\in V''_+} |\cz_y|.\label{eq:sumzV2}
\end{align}
Now we apply again the lower bound of $|\lambda_n|\geq \sqrt{2n-4}-\frac{3}{2}-O(\frac{1}{\sqrt{n}})$. Simplifying Equation \eqref{eq:sumzV2}, we get
\begin{equation}
   \left(3\sqrt{2n-4} +  O\left(1\right)\right)  \sum_{v\in V_+}|\cz_v| \geq  10\sqrt{2n-4} \sum_{v\in V''_+}|\cz_v|.
\end{equation}
Thus, we have
\begin{equation}
    \vol(V''_+) \leq\left(\frac{3}{10} +  O\left(\frac{1}{\sqrt{n}}\right)\right) \vol(V_+).
\end{equation}
Equivalently,
\begin{equation}\label{eq:V''}
    \vol(V''_+) \leq\left( \frac{3}{7}+  O\left(\frac{1}{\sqrt{n}}\right)\right) \vol(V'_+).
\end{equation}
As a corollary, we have
\begin{equation}
    \vol(V_+) =\vol(V'_+) + \vol(V''_+)\leq 
     \lp \frac{10}{7}+  O\left(\frac{1}{\sqrt{n}}\right)\rp \vol(V'_+).
\end{equation}
\end{proof}

Using Lemma \ref{lem:V+}, we deduce some important information on the structure of the extremal graph.

\begin{lemma}\label{lem:z-deg}
We have
    \begin{enumerate}[(i)]
        \item There exist $v_1, v_2 \in V_+$ with $\min\{d(v_1), d(v_2)\} \geq n-5\sqrt{2n-4}$.
        \item $w_0 \in \{v_1,v_2\}$. 
        \item For all $v \in V(G)\backslash \{v_1, v_2\}$, $d(v)\leq 10\sqrt{2n-4}+4$.
        \item For all $v \in V(G)\backslash \{v_1, v_2\}$, $|\cz_v|=O(\frac{1}{\sqrt{n}}).$
        \item Assume $w_0=v_1$. Then $\cz_{v_2}\geq 1-O(\frac{1}{\sqrt{n}})$.
    \end{enumerate}
\end{lemma}

\begin{proof}
Observe that $|V'_+|\leq 2$. Otherwise, any three vertices in  $|V'_+|$ have  common neighbors of size at least $|V_-| - 15\sqrt{2n-4}\geq 5\sqrt{2n-4} >3 $.
Thus $G$ contains a subgraph $K_{3,3}$, contradicting that $G$ is planar.
Thus, we have that for sufficiently large $n$,
$$\vol(V_{+}') \leq |V'_+|\leq 2.$$
Hence by Lemma \ref{lem:V+}, we have 
$$\vol(V_{+}'')  \leq\left( \frac{3}{7}+  O\left(\frac{1}{\sqrt{n}}\right)\right)\cdot 2 <1.$$
This implies that $w_0\not\in V''_+$ since $|\cz_{w_0}|=1$. 
 Moreover $w_0 \notin V_{-}$, as otherwise 
$$\sqrt{2n-4}-\frac{3}{2}-O\lp\frac{1}{\sqrt{n}} \rp\leq |\lambda_n| = \lambda_n \cz_{w_0} \leq \dss_{u\in V_{+} \cap N(w_0)} \cz_u  \leq \vol(V_{+}) <3,$$
giving a contradiction.
Thus $w_0 \in V'_+$. In particular, $z_{w_0}=1$.

Now we show $V'_+$ has exactly two vertices. If not, assume $w_0$ is the only vertex in $V'_+$. We have
\begin{align*}
    |\lambda_n|^2 &=  |\lambda_n|^2 \cz_{w_0} \\
    &\leq |\lambda_n|  \sum_{u\in N(w_0)\cap V_-} |\cz_u|\\
    &\leq \sum_{u\in N(w_0)\cap V_-} \sum_{y\in N(u)\cap V_+} \cz_y\\
    &\leq n + (2n-4) \sum_{y\in V''_+} \cz_y\\
    &\leq n + (2n-4) \left( \frac{3}{7}+  O\left(\frac{1}{\sqrt{n}}\right)\right) \\
    &\leq (2-\frac{1}{7})n + O(\sqrt{n}),
\end{align*}
contradicting the lower bound of $|\lambda_n|$. Hence $|V_+'|=2$. Let $V_+' = \{v_1, v_2\}$.
Notice that $d(v)\leq 10\sqrt{2n-4}+4$ for any $v\not= v_1, v_2$. Otherwise $v, v_1, v_2$ have a common neighborhood of size at least $3$, contradicting that $G$ is $K_{3,3}$-free.

Now we will show that for any $v \notin \{v_1, v_2\}$, $|\cz_v| = O\left(\frac{1}{\sqrt{n}}\right)$.
For any $v\in V''_+$, we have
\begin{align*}
    |\lambda_n|^2 \cz_v &\leq |\lambda_n|  \sum_{u\in N(v)\cap V_-} |\cz_u|\\
      &\leq \sum_{u\in N(v)\cap V_-} \sum_{y\in N(u)\cap V_+} \cz_y\\
      &= \sum_{y\in V_+}\cz_y\cdot |N(v)\cap N(y)\cap V_-| \\
      &\leq \left( 10\sqrt{2n-4}+4 \right) \sum_{y\in V_+}\cz_y \\
      &\leq \left( 10\sqrt{2n-4}+4 \right) 2\left(\frac{10}{7} +  O\left(\frac{1}{\sqrt{n}}\right) \right).
      \end{align*}
Thus, $\cz_v= O\left(\frac{1}{\sqrt{n}}\right)$.

For $u\in V_-$, applying Equation \eqref{eq:zu}, we have
\begin{align*}
    |\lambda_n| \cz_u &\leq \sum_{v\in N(v)\cap V_+} \cz_v\\
    &\leq  \sum_{v\in V_+} \cz_v\\
    &\leq 2 \lp \frac{10}{7} +  O\left(\frac{1}{\sqrt{n}}\right)\rp.
\end{align*}
Therefore, $\cz_u= O\left(\frac{1}{\sqrt{n}}\right)$.

Finally, we estimate $\cz_{v_2}$. 
From the eigen-equations, we get
\begin{align}
    |\lambda_n| (\cz_w -\cz_{v_2}) & = -\sum_{u\in N(w)\setminus N(v_2)} \cz_u + \sum_{u\in N(v_2)\setminus N(w)} \cz_u \\
    &\leq \sum_{u\in (N(w)\setminus N(v_2))\cap V_-} |\cz_u| +  \sum_{u\in (N(v_2)\setminus N(w))\cap V_+} \cz_u\\
    &\leq \sum_{u\in (N(w)\setminus N(v_2))\cap V_-} |\cz_u|  + \sum_{u\in V''_+} \cz_u\\
    &\leq 10\sqrt{2n-4} \cdot O\left(\frac{1}{\sqrt{n}}\right) +  \frac{6}{7} +  O\left(\frac{1}{\sqrt{n}}\right) \\
    &= O(1).
\end{align}
Therefore, we have
$\cz_{v_2}\geq 1 - O\left(\frac{1}{\sqrt{n}}\right).$
\end{proof}

For $i\in \{0,1,2\}$, let $V_i = \{v \in V(G)\backslash \{v_1, v_2\}: N(v) \cap \{v_1, v_2\}= i\}$. 
We have the following lemma on the structure of $G$.
\begin{lemma}\label{lem:stucture_G}
We have the following properties.
\begin{enumerate}[(i)]
     \item $|V_2|\geq n - 10\sqrt{2n-4}.$
    \item For any $v\in V_0\cup V_1\cup V_2$, $|N(v)\cap V_2| \leq 2$.
    \item In $H= G[V_0\cup V_1\cup V_2]$, for any vertex $v \in V(H)$,
    $|N_{H}(N_{H}(v))\cap V_2|\leq 4$.
\end{enumerate}
\end{lemma}
\begin{proof}
By Lemma \ref{lem:z-deg}, $\min\{d(v_1), d(v_2)\}\geq n-5\sqrt{2n-4}$. It follows that $|V_2|\geq n-10\sqrt{2n-4}$. For any $v\in V_0 \cup V_1\cup V_2$, $v$ has at most two neighbors in $V_2$, otherwise, $v, v_1, v_2$ and three of their common neighbors would form a $K_{3,3}$ in $G$.

Now for any $v \in V(G[V_0\cup V_1\cup V_2])$, we claim that $|N_{H}(N_{H}(v))\cap V_2|\leq 4$. Indeed, suppose not, then by (ii) and the Pigeonhole principle, there exist three vertex-disjoint $2$-vertex paths $u_1 w_1, u_2 w_2, u_3 w_3$ in $H$ such that $v$ is adjacent to $u_1, u_2, u_3$, and $w_1, w_2, w_3 \in N(N(v)) \cap V_2$. We then have a $K_{3,3}$ minor in $G$, contradicting that $G$ is planar.
\end{proof}

Using Lemma \ref{lem:stucture_G}, we can obtain bounds on the entries of $\cx$. 
\begin{lemma}\label{lem:x}
Let $u_0$ be the vertex such that $\cx_{u_0} = 1$.
\begin{enumerate}[(i)]
        \item $u_0 \in \{v_1, v_2\}$.
        \item $\min\{\cx_{v_1}, \cx_{v_2}\} \geq 1 - O(\frac{1}{\sqrt{n}})$.
        \item For any other vertex $v \notin \{v_1, v_2\}$, $\cx_v  = O(\frac{1}{\sqrt{n}})$.
    \end{enumerate}
\end{lemma}

\begin{proof}
Let prove (iii) first.
For any vertex $v \notin \{v_1, v_2\}$, we have
\begin{align*}
    \lambda_1^2 \cx_v &= \lambda_1 \sum_{s\in N(v)}\cx_s \\
    &= \lambda_1 \left(\sum_{s\in N(v)\cap V_2}\cx_s + \sum_{s\in N(v)\cap \{v_1, v_2\}} \cx_s + \sum_{s\in N(v)\cap (V_0\cup V_1)} \cx_s     \right)\\ 
    &\leq 4\lambda_1 +  \sum_{s\in N(v)\cap (V_0\cup V_1)} \lambda_1 \cx_s\\
      &=  4\lambda_1 + \sum_{s\in N(v)\cap (V_0\cup V_1)} \sum_{t\in N(s)}\cx_t\\
       &=  4\lambda_1 + \sum_{s\in N(v)\cap (V_0\cup V_1)} \left(\sum_{t\in N(s)\cap \{v_1,v_2\}}\cx_t +
       \sum_{t\in N(s)\cap V_2}\cx_t +\sum_{t\in N(s)\cap (V_0\cup V_1)}\cx_t 
       \right)\\
    &\leq   4\lambda_1 +  4 |N(v)\cap V_0\cup V_1)| + 4 + \sum_{s\in N(v)\cap (V_0\cup V_1)} \sum_{t\in N(s)\cap (V_0\cup V_1)}\cx_t\\
    &\leq  4\lambda_1 +  4|V_0\cup V_1| + 4 +   2 |E(G[V_0\cup V_1])|\\
     &\leq  4\lambda_1 + 4 |V_0\cup V_1| + 4 +   2( 3|V_0\cup V_1|-6)\\
     &\leq  4\lambda_1 + 10 |V_0\cup V_1|\\
    &= O(\sqrt{n}). 
\end{align*}
We conclude that
$$\cx_v = O\left( \frac{1}{\sqrt{n}}\right).$$
Thus, $u$ must be one of $v_1$ or $v_2$.

If $v_1v_2$ is not an edge of $G$, then we have
\begin{align*}
    \lambda_1|\cx_{v_1}-\cx_{v_2}| &\leq \sum_{v\in V_1}\cx_v \\
    &\leq |V_1|\cdot O\left( \frac{1}{\sqrt{n}}\right)\\
    &= O(1).
\end{align*}
If $v_1v_2$ is an edge of $G$, we have
\begin{align*}
    (\lambda_1-1)|\cx_{v_1}-\cx_{v_2}| &\leq \sum_{v\in V_1}\cx_v \\
    &\leq |V_1|\cdot O\left( \frac{1}{\sqrt{n}}\right)\\
    &= O(1).
\end{align*}
In both cases, we have
$$ |\cx_{v_1}-\cx_{v_2}| =   O\left( \frac{1}{\sqrt{n}}\right).$$  
It follows that $\min\{\cx_{v_1}, \cx_{v_2}\} \geq 1 - O(\frac{1}{\sqrt{n}})$.
\end{proof}

In the next lemma, we show that the extremal planar graph attaining the maximum spread must contain $K_{2,n-2}$ as a subgraph.

\begin{lemma}\label{lem:degn-2}
    Let $G$ be a graph obtaining the maximum spread among all $n$-vertex planar graphs. Then there exist two vertices $v_1, v_2$ in $G$ such that each of $v_1, v_2$ is adjacent to all vertices in $V\backslash \{v_1, v_2\}$. 
\end{lemma}
\begin{proof}
    Let $\cx$ and $\cz$ be the eigenvectors associated with $\lambda_1$ and $\lambda_n$ respectively. Assume that $\cx$ and $\cz$ are both normalized such that the largest entries of them in absolute value are $1$.
    By Lemma \ref{lem:z-deg}, there exist two vertices $v_1, v_2 \in V_+$ such that $\min\{d_{v_1}, d_{v_2}\} \geq n - 5\sqrt{2n-4}$. Recall that for $i\in \{0,1,2\}$, $V_i = \{v \in V(G)\backslash \{v_1, v_2\}: N(v) \cap \{v_1, v_2\}= i\}$. 
    
    It suffices to show that $V_0 \cup V_1$ is empty. Suppose otherwise that $V_0 \cup V_1$ is not empty. Since $V_0 \cup V_1$ induces a planar graph, there exists some vertex $v \in V_0 \cup V_1$ such that $|N(v)\cap (V_0 \cup V_1)| \leq 5$. Moreover, observe that $v$ has at most two neighbors in $V_2$, as otherwise $v, v_1, v_2$ and three of their common neighbors would form a $K_{3,3}$ in $G$.
    Let $G'$ be obtained from $G$ by removing all the edges of $G$ incident with $v$ and adding the edges $vv_1, vv_2$, so that $E(G') = E(G-v) \cup \{vv_1, vv_2\}$.  Observe $G'$ is still planar.

    We claim that $\lambda_n(G') < \lambda_n(G)$. Indeed, consider the vector $\tilde{\cz}$ such that $\tilde{\cz}_{u} = \cz_u$ for $u\neq v$ and $\tilde{\cz}_v = -|\cz_v|$. Then
    \begin{align*}
        \tilde{\cz}' A(G') \tilde{\cz} &\leq \cz' A(G) \cz + 2 \dss_{y\sim v} |\cz_y \cz_v| - 2|\cz_v| (\cz_{v_1} + \cz_{v_2})\\
                      & \leq \cz' A(G)\cz + 2 \cdot (2 + 5) \cdot O\left(\frac{1}{\sqrt{n}}\right)\cdot |\cz_v| -  \lp 1- O\left(\frac{1}{\sqrt{n}}\right)\rp |\cz_v|\\
                      & < \cz' A(G) \cz.
    \end{align*}   

   Similarly, we claim that $\lambda_1(G') > \lambda_1(G)$. Indeed,
     \begin{align*}
       \cx'\cx \lambda_n(G') &= \cx' A(G') \cx \\
                      &= \cx' A(G) \cx - 2 \dss_{y\sim v} \cx_y \cx_v + 2\cx_v (\cx_{v_1} + \cx_{v_2})\\
                      & \geq \cx' \cx \lambda_n(G) - 2 \cdot (2 + 5) \cdot O\left(\frac{1}{\sqrt{n}}\right)\cdot \cx_v +  \lp 1- O\left(\frac{1}{\sqrt{n}}\right)\rp \cx_v\\
                      & > \cx' \cx \lambda_n(G)
    \end{align*}
Hence we have $S(G') =\lambda_1(G') -\lambda_n(G') > \lambda_1(G) -\lambda_n(G) = S(G)$, giving a contradiction.
     
\end{proof}

\begin{corollary}\label{cor:planarstruct}
For sufficiently large $n$, the planar graph on $n$ vertices attaining the maximum spread must belong to one of the three families below.
\begin{enumerate}[(i)]
\item $G$ is a double wheel $(K_1 \cup K_1)\vee C_{n-2}$.
\item $G$ is $(K_1 \cup K_1)\vee T_{n-2}$, where $T_{n-2}$ is a linear forest on $n-2$ vertices.
\item $G$ is $K_2\vee T_{n-2}$, where $T_{n-2}$ is a linear forest on $n-2$ vertices.
\end{enumerate}
\end{corollary}
\begin{proof}
By Lemma~\ref{lem:degn-2}, for sufficiently large $n$, the maximum-spread planar graph $G$ on $n$ vertices has two vertices $v_1$ and $v_2$ which are adjacent to every vertex in $A = V(G)\setminus \{v_1, v_2\}$. If there is a vertex $u \in A$ with $|N(u) \cap A| \ge 3$, then $G$ has $K_{3, 3}$ as a subgraph. It follows that the $G[A]$ has maximum degree at most $2$ and hence is a disjoint union of paths and cycles. If $v_1 v_2 \in E(G)$, and $G[A]$ contains a cycle $C$, then $G[\{v_1, v_2\} \cup C]$ contains a $K_5$ minor, contradicting that $G$ is planar. Hence, $G[A]$ must be a linear forest on $n-2$ vertices, proving (iii). Now suppose $v_1 v_2 \notin E(G)$. If $G[A]$ is a cycle, then $G$ is the double wheel $(K_1 \cup K_1) \vee C_{n-2}$, proving (i).
Otherwise, if $G[A]$ contains a cycle $C$ which does not span all the vertices of $A$, then by contracting edges and deleting vertices in the subgraph induced by $A\setminus C$, we obtain that $G$ contains $K_2 \vee C$ as a minor, and therefore has $K_5$ as a minor, giving a contradiction. Hence, if $G[A]$ is not $C_{n-2}$, it must be a linear forest, proving (ii).
\end{proof}

We call $(K_1 \cup K_1)\vee T_{n-2}$ a {\em linear planar graph of the first kind} and  $K_2\vee T_{n-2}$ a {\em linear planar graph of the second kind}. 
By Corollary \ref{cor:planarstruct}, the maximum spread on planar graphs is achieved by either a double wheel, or a linear planar graph of the first kind, or a linear planar graph of the second kind. In the next few subsections, we will compute the maximum spread of graphs in these three families respectively, and then obtain the extremal graph attaining the maximum spread among all planar graphs. 

\subsection{Double wheel graph}
We first treat the first case of Corollary~\ref{cor:planarstruct}, the double wheel $(K_1 \cup K_1) \vee C_{n-2}$.
\begin{lemma}\label{lem:spreadwheel}
The spread of the double wheel graph $(K_1 \cup K_1) \vee C_{n-2}$ is \[\sqrt{8n-12} = 2\sqrt{2(n-2)}+\frac{1}{\sqrt{2(n-2)}} + O\left(\frac{1}{(2(n-2))^{\frac32}}\right).\]
\end{lemma}

\begin{proof}
By Lemma~\ref{lem:joinreglemma}, for $n$ sufficiently large, the largest and smallest eigenvalue of the graph $(K_1 \cup K_1) \vee C_{n-2}$ are $1+\frac{\sqrt{8n-12}}{2}$ and $1 -\frac{\sqrt{8n-12}}{2}$ respectively. Thus the spread is $\sqrt{8n-12}$.
We now give an asymptotic expansion for $\sqrt{8n-12}$ using the Taylor expansion for $(1+x)^{\frac12}$. We have
\begin{equation}
\begin{aligned}
    \sqrt{8n-12}&= \sqrt{(8n-16)+4}\\
    &=\sqrt{8n-16}\left(\sqrt{1+\frac{4}{8n-16}}\right)\\
    &=\sqrt{8n-16}\left(1+\frac12\left(\frac{4}{8n-16}\right)-\frac18\left(\frac{4}{8n-16}\right)^2+O\left(\frac{1}{(8n-16)^3}\right)\right)\\
    &=2\sqrt{2(n-2)}+\frac{1}{\sqrt{2(n-2)}}-\frac{1}{4(2(n-2))^{\frac32}}+O\left(\frac{1}{(8n-16)^{\frac52}}\right).
\end{aligned}
\label{eqn:wheelasymp}
\end{equation}
\end{proof}

\subsection{Linear planar graphs of the first kind}
Consider a linear planar graph of the first kind $G = (K_1 \cup K_1) \vee (P_{\ell_1} \cup P_{\ell_2} \cup \cdots P_{\ell_r})$. Let $u,w$ be the center vertices of $G$ and $v_1, \ldots, v_\ell$ be the vertices in order of a path component in $G - u - w$.  Note that there is a graph automorphism $\phi$ that maps $u$
to $w$. 
Let $\alpha$ be a normalized eigenvector (invariant under $\phi$) corresponding to an eigenvalue $\lambda$ of the adjacency matrix of $G$ so that $\alpha(u)=\alpha(w)=1$. 
 Let $A_\ell$ be the adjacency matrix of $P_\ell$, $x_i=\alpha(v_i)$ for $1\leq i \leq \ell$, and $\bx=(x_1,\ldots, x_\ell)'\in {\mathbb R}^{\ell}$.

Following along the same lines as the outerplanar case, we have the following lemmas and corollaries.
\begin{lemma} \label{l6}
    Let $G= (K_1 \cup K_1) \vee (P_{\ell_1} \cup P_{\ell_2} \cup \cdots \cup P_{\ell_r})$ be a linear planar graph on $n$ vertices with center vertices $u$ and $w$. Suppose $\lambda$ is an eigenvalue of $A(G)$ with $|\lambda|\geq 2$, and $\alpha$ is a normalized eigenvector of $A(G)$ corresponding to $\lambda$ such that $\alpha(u)=\alpha(w) = 1$. 
    Let $P_{\ell}$ be one of the path components of $G - u-w$ and let ${\bf x}$ and $A_{\ell}$ be defined as above.
    Then
\begin{equation}
    {\bf x}=2\sum_{k=0}^\infty \lambda^{-(k+1)} A_\ell^k \one.
\end{equation}
\end{lemma}
The proof is similar to the proof for Lemma \ref{l1}, but instead of $A_\ell \mathbf{x} + \one = \lambda \mathbf{x}$ in equation \ref{eq:c5}, we have $A_\ell \mathbf{x} + 2\cdot \one = \lambda \mathbf{x}$ in the case of two center vertices. The results below follow correspondingly.
\begin{corollary}\label{c2}
The following properties hold for any linear planar graph.
\begin{enumerate}
    \item If $\lambda\geq 2$, all entries of $\alpha$ are positive.
    \item If $\lambda\leq -2$,  all entries of $\alpha$ but the $u$-entry are negative.
\end{enumerate}
\end{corollary}
Consider a linear planar graph $G = H \vee (P_{\ell_1} \cup P_{\ell_2} \cup \cdots \cup P_{\ell_r})$ with $\ell_1, \ell_2 \geq 2$, and $H \in \{K_1\cup K_1, K_2\}$. Similar as in the case for outerplanar graphs, the merge operation replaces $G$ by $G' = H \vee (P_{\ell_1 + \ell_2 - 1} \cup P_{1} \cup \cdots \cup P_{\ell_r})$. Following along the same lines of the outerplanar case, we have that the merge operation increases the spread of a linear planar graph.
\begin{lemma} \label{lem:merge-planar}
  Let $G= H \vee (P_{\ell_1} \cup P_{\ell_2} \cup \cdots \cup P_{\ell_r})$ be a linear planar graph on $n\geq 8$ vertices with $\ell_1, \ell_2 \geq 2$ and $H \in \{K_1\cup K_1, K_2\}$. Let $G'$ be obtained from $G$ by applying a merge operation, i.e., $G'=H \vee (P_{\ell_1+\ell_2-1} \cup P_1\cup \cdots \cup P_{\ell_r})$  Then $S(G') > S(G)$. 
\end{lemma}
The proofs are nearly identical to the proofs for Lemma \ref{l2} and \ref{l3}. We proceed to the main theorem for linear planar graphs of the first kind.
By Lemma \ref{lem:merge-planar}, we only need to consider the linear planar graph with only one non-trivial path. Let $G'_\ell = (K_1 \cup K_1) \vee (P_\ell \cap (n-2 - \ell)K_1)$. 

\begin{lemma}\label{lemma:spreadplanar1}
For sufficiently large $n$, the spread of $G'_\ell$ is given by
$$S(G'_\ell) = \lambda_1 - \lambda_n = 2\sqrt{2(n-2)} + \frac{2c'_2}{\sqrt{2(n-2)}} + \frac{2c'_4}{(2(n-2))^{\frac{3}{2}}} + \frac{2c'_6}{(2(n-2))^\frac{5}{2}} + O\left(\frac{1}{(2(n-2))^3}\right).$$
Here,
\begin{align*}
    c'_2 &= -\frac{3}{2}\left(\frac{l-2}{n-2}\right)^2 + 2\left(\frac{l-2}{n-2}\right),\\
    c'_4 &= -\frac{105}{8}\left(\frac{l-2}{n-2}\right)^4 + 35\left(\frac{l-2}{n-2}\right)^3 - 30\left(\frac{l-2}{n-2}\right)^2 + 2\left(\frac{l-2}{n-2}\right) + 2, \\
    c'_6 &= -\frac{3003}{16}\left(\frac{l-2}{n-2}\right)^6 + \frac{3003}{4}\left(\frac{l-2}{n-2}\right)^5 - 1155\left(\frac{l-2}{n-2}\right)^4 + 735\left(\frac{l-2}{n-2}\right)^3, \\
    &- 105\left(\frac{l-2}{n-2}\right)^2 - 28\left(\frac{l-2}{n-2}\right) - 12.
\end{align*}
\end{lemma}

\begin{proof}
 Let $\lambda$ be either $\lambda_1$ or $\lambda_n$. Let $\alpha$ be the normalized eigenvector associated with $\lambda$ such that $\alpha(u) = \alpha(w) = 1$, where $u,w$ are the center vertices in $G_\ell'$. Let $P_{\ell}$ be the unique non-trivial path in $G-u-w$ and $\mathbf{x} = (x_1, \ldots, x_\ell)$ be the vector of $\alpha$ restricted to the vertices of $P_\ell$. Let $v$ be a vertex in $G_\ell'$ that is not on $P_\ell$ and is not $u$ or $w$. The eigen-equation on $v$ gives that
$\lambda \alpha(v) = \sum_{y \sim v} \alpha(y) = \alpha(u) + \alpha(w) = 2$,
so $\alpha(v) = \frac{2}{\lambda}$. The eigen-equation at $u$ (and $w$) gives
\begin{equation} \label{e19}
    \lambda = \sum_{i = 1}^\ell x_i + (n - 2 - \ell) \frac{2}{\lambda}.
\end{equation}
Applying Lemma \ref{l6}, we get
\begin{align*}
    \sum_{i=1}^\ell x_i &= \one' \cdot \bx\\
    &= \one' \cdot 2\sum_{k=0}^\infty \lambda^{-(k+1)} A_\ell^k \one \\
    &= 2\sum_{k=0}^\infty \lambda^{-(k+1)} \one'A_\ell^k \one. 
\end{align*}
Plugging into Equation (\ref{e19}), we have
\begin{equation} \label{e20}
    \lambda = (n - 2 - \ell) \frac{2}{\lambda} + 2\sum_{k=0}^\infty \lambda^{-(k+1)} \one' A_\ell^k \one.
\end{equation}
By a similar argument in the proof of Theorem \ref{main}, we get
\begin{equation} \label{e21}
    \lambda^2 = 2n - 4 + \frac{4\ell - 4}{\lambda} + \frac{8\ell - 12}{\lambda^2} + \frac{16\ell - 32}{\lambda^3} + \frac{32\ell - 76}{\lambda^4} + \frac{64\ell - 176}{\lambda^5} + O\left(\frac{\ell}{\lambda^6}\right).
\end{equation}
This equation has two real roots which determines $\lambda_1$ and $\lambda_n$. $\lambda_1$ is near $\sqrt{2(n-2)}$ and $\lambda_n$ is near $-\sqrt{2(n-2)}$. By Lemma~\ref{lem:laurent} in the appendix, $\lambda$ has the following series expansion:
\begin{equation} \label{e22}
    \lambda_1 = \sqrt{2(n-2)} + c'_1 + \frac{c'_2}{\sqrt{2(n-2)}} + \frac{c'_3}{2(n-2)} + \frac{c'_4}{(2(n-2))^{\frac{3}{2}}} + \frac{c'_5}{(2(n-2))^2} + \frac{c'_6}{(2(n-2))^\frac{5}{2}} + O\left(\frac{1}{(2(n-2))^3}\right).
\end{equation}
Calculating this out in SageMath, we get

\begin{align*}
    c'_1 &= \frac{l-2}{n-2}, \\
    c'_2 &= -\frac{3}{2}\left(\frac{l-2}{n-2}\right)^2 + 2\left(\frac{l-2}{n-2}\right),\\
    c'_3 &= 4\left(\frac{l-2}{n-2}\right)^3 - 8\left(\frac{l-2}{n-2}\right)^2 + 4\left(\frac{l-2}{n-2}\right) + 2, \\
    c'_4 &= -\frac{105}{8}\left(\frac{l-2}{n-2}\right)^4 + 35\left(\frac{l-2}{n-2}\right)^3 - 30\left(\frac{l-2}{n-2}\right)^2 + 2\left(\frac{l-2}{n-2}\right) + 2, \\
    c'_5 &= 48\left(\frac{l-2}{n-2}\right)^5 -160\left(\frac{l-2}{n-2}\right)^4 + 192\left(\frac{l-2}{n-2}\right)^3 - 72\left(\frac{l-2}{n-2}\right)^2 - 8\left(\frac{l-2}{n-2}\right), \\
    c'_6 &= -\frac{3003}{16}\left(\frac{l-2}{n-2}\right)^6 + \frac{3003}{4}\left(\frac{l-2}{n-2}\right)^5 - 1155\left(\frac{l-2}{n-2}\right)^4 + 735\left(\frac{l-2}{n-2}\right)^3 \\
    &- 105\left(\frac{l-2}{n-2}\right)^2 - 28\left(\frac{l-2}{n-2}\right) - 12.
\end{align*}

Similarly, by Lemma~\ref{lem:laurent} in the appendix, we get the following series expansion of $\lambda_n$:
\begin{equation} \label{e23}
    \lambda_n = -\sqrt{2(n-2)} + c'_1 - \frac{c'_2}{\sqrt{2(n-2)}} + \frac{c'_3}{2(n-2)} - \frac{c'_4}{(2(n-2))^{\frac{3}{2}}} + \frac{c'_5}{(2(n-2))^2} - \frac{c'_6}{(2(n-2))^\frac{5}{2}} + O\left(\frac{1}{(2(n-2))^3}\right).
\end{equation}
Here, $c'_1,c'_2,c'_3,c'_4,c'_5,c'_6$ are the same as before. 
Observe that all $c'_i$s are bounded since they are polynomials of $\frac{\ell-2}{n-2}$, which is contained in the interval $[0,1]$. Thus, the spread of $G'_\ell$, denoted by $S(G'_\ell)$, can be expressed as:
  $$  S(G'_\ell) = \lambda_1 - \lambda_n = 2\sqrt{2(n-2)} + \frac{2c'_2}{\sqrt{2(n-2)}} + \frac{2c'_4}{(2(n-2))^{\frac{3}{2}}} + \frac{2c'_6}{(2(n-2))^\frac{5}{2}} + O\left(\frac{1}{(2(n-2))^3}\right).$$
\end{proof}

\subsection{Linear planar graphs of the second kind}
Consider the linear planar graphs of the second kind $G = K_2 \vee (P_{\ell_1} \cup P_{\ell_2} \cup \cdots P_{\ell_r})$. By Lemma \ref{lem:merge-planar}, the maximum spread among all linear planar graphs of the second kind can only be achieved by
$G''_\ell= K_2 \vee (P_{\ell} \cup (n-\ell-2) P_1)$.
Using the same method from Theorem~\ref{main} and Lemma~\ref{lemma:spreadplanar1}, we obtain that the spread of $G''_\ell$ is
\[S(G''_\ell) = 2\sqrt{2(n-2)} + \frac{2c''_2}{\sqrt{2(n-2)}} + O\left(\frac{1}{(2(n-2))^{\frac32}}\right),
\]
where $c''_2 = -\frac{3}{2}\left(\frac{l-2}{n-2}\right)^2 + \frac{3}{2}\left(\frac{l-2}{n-2}\right) + \frac{1}{8}$. 

\subsection{Proof of the main theorem}
\begin{lemma}\label{lem:firstkind}
   For sufficiently large $n$, a maximum-spread planar graph on $n$ vertices is a linear planar graph of the first kind.
\end{lemma}
\begin{proof}
Let $G_{\ell}'$ and $G_{\ell}''$ denote the linear planar graph of the first kind and second kind, respectively, such that the non-center vertices induce a linear forest with  a unique non-trivial path $P_{\ell}$.
Let $\ell_0 = \lceil \frac{2n - 2}{3} \rceil$ be the target maximum point of $S(G'_\ell)$. We have
\begin{equation}
    S(G'_{\ell_0}) = 2\sqrt{2(n-2)} + \frac{4}{3\sqrt{2(n-2)}} + O\left(\frac{1}{(2(n-2))^{3/2}}\right).
\end{equation}
\label{eqn:fspread}
For linear planar graphs of the second kind, observe that $c''_2\leq \frac{1}{2}$. Thus for sufficiently large $n$ and for any $\ell \in [n]$,
\begin{equation}\label{eqn:sfspread}
S(G''_\ell) \leq  2\sqrt{2(n-2)} +\frac{1}{\sqrt{2(n-2)}} + O\left(\frac{1}{(2(n-2))^{\frac32}}\right) < S(G'_{\ell_o}).
\end{equation}
Hence for sufficiently large $n$, the maximum spread of linear planar graphs of the second kind is less than
the maximum spread of linear planar graphs of the first kind.
Similarly, for the double wheel graph, by Lemma~\ref{lem:spreadwheel}, we have
\[S((K_1 \cup K_1)\vee C_{n-2}) = 2\sqrt{2(n-2)} + \frac{1}{\sqrt{2(n-2)}} + O\left(\frac{1}{(2(n-2))^{\frac32}}\right) < S(G_{\ell_0}').\]
This completes the proof of the lemma.
\end{proof}
We proceed with the proof of the main theorem for planar graphs.
\begin{proof}[Proof of Theorem~\ref{mainplanar}]
By Lemma~\ref{lem:firstkind}, the maximum-spread planar graph is a linear planar graph of the first kind with at most one non-trivial path.
Recall that the spread equation $S(G'_\ell)$ for linear planar graphs of the first kind is given by Theorem~\ref{lemma:spreadplanar1}.
Since
\begin{align*}
    c'_2(\ell) &= -\frac{3}{2}\left(\frac{\ell-2}{n-2}\right)^2 + 2\left(\frac{\ell-2}{n-2}\right) \\
    &= \frac{2}{3} - \frac{3}{2}\left(\frac{\ell-2}{n-2} - \frac{2}{3}\right)^2 \\
    &\leq \frac{2}{3},
\end{align*}
the function $c'_2$ reaches the maximum at $\ell_1 = \frac{2(n-2)}{3} + 2 = \frac{2n - 1}{3} + 1$. 
Let $\ell_0 = \lceil \frac{2n - 2}{3} \rceil$ be the target maximum point of $f(\ell)$. 

We claim that there is a constant $C$ such that all possible maximal points must be contained in the interval $(\ell_1 - C\sqrt{2(n-2)}, \ell_1 + C\sqrt{2(n-2)})$. 
Otherwise, for any $\ell$ not in this interval, we have
$$c'_2(\ell)\leq \frac{2}{3}-\frac{3C^2}{2(n-1)}.$$
This implies
\[
    S(G'_\ell) \leq 2\sqrt{2(n-2)} + 2\frac{\frac{2}{3}-\frac{3C^2}{2(n-1)}}{\sqrt{2(n-2)}} + O\left(\frac{1}{(2(n-2))^{3/2}}\right) < S(G'_{\ell_0}).
\]
Here, we choose $C$ big enough  such that
\[
    -\frac{3C^2}{\sqrt{2}(n-2)^{3/2}} + O\left(\frac{1}{(2(n-2))^{3/2}}\right) < 0,
\]
which proves the claim. \\
From now on, we can assume $\ell \in (\ell_1 - C\sqrt{2(n-2)}, \ell_1 + C\sqrt{2(n-2)})$. Next, we compute $S(G'_{\ell+1}) - S(G'_{\ell}).$ We have
\begin{align*}
    c'_2(\ell + 1) - c'_2(\ell) &= -\frac{3}{2}\left(\frac{2\ell - 3}{(n - 2)^2}\right) + 2\left(\frac{1}{n - 2}\right) = \frac{2(n - 2) - \frac{3}{2}(2\ell - 3)}{(n - 2)^2}, \\
    c'_4(\ell + 1) - c'_4(\ell) &= -\frac{105}{8}\left(\frac{4\ell^3 - 18\ell^2 + 28\ell - 15}{(n-2)^4}\right) + 35\left(\frac{3\ell^2 - 9\ell + 7}{(n-2)^3}\right) - 30\left(\frac{2\ell - 3}{(n-2)^2}\right) + 2\left(\frac{1}{n-2}\right), \\
    c'_6(\ell + 1) - c'_6(\ell) &= O\left(\frac{1}{n-2}\right).
\end{align*}
Plugging in $\ell = \ell_1 \cdot \left(1 + O\left(\frac{1}{\sqrt{2(n-2)}}\right)\right)$ into $c'_4(\ell + 1) - c'_4(\ell)$, we have
\begin{align*}
    c'_4(\ell + 1) - c'_4(\ell) &= \frac{1}{n - 2} \left( -\frac{105}{8} \cdot 4 \cdot \left( \frac{2}{3} \right)^3 + 35 \cdot 3 \left(\frac{2}{3}\right)^2 - 30 \cdot 2 \cdot \left(\frac{2}{3}\right) + 2 + O\left( \frac{1}{\sqrt{2(n-2)}} \right) \right)\\
    &= -\frac{62}{9(n-2)} + O\left(\frac{1}{\sqrt{2}(n-2)^{3/2}}\right).
\end{align*}
Therefore, we have
\begin{align*}
    S(G'_{\ell + 1}) - S(G'_\ell) &= \frac{2(c'_2(\ell + 1) - c'_2(\ell))}{\sqrt{2(n-2)}} + \frac{2(c'_4(\ell + 1) - c'_4(\ell))}{(2(n-2))^{\frac{3}{2}}} + \frac{2(c'_6(\ell + 1) - c'_6(\ell))}{(2(n-2))^\frac{5}{2}} + O\left(\frac{1}{(2(n-2))^3}\right) \\
    &= \frac{4(n - 2) - 3(2\ell - 3)}{\sqrt{2}(n - 2)^{5/2}} -\frac{62}{9\sqrt{2}(n-2)^{5/2}} + O\left( \frac{1}{(n-2)^3}\right) \\
    &= \frac{4n - 6\ell - \frac{53}{9}}{\sqrt{2}((n-2))^{5/2}} + O\left( \frac{1}{(n-2)^3}\right).
\end{align*}
When $\ell \geq \ell_0$, we have
\[
    4n - 6\ell - \frac{53}{9} \leq 4n - 6\ell_0 - \frac{53}{9} \leq 4n - 6 \cdot \frac{2n - 2}{3} - \frac{53}{9} = -\frac{17}{9} < 0.
\]
It follows that $S(G'_{\ell + 1}) - S(G'_\ell) < 0$. \\
When $\ell \leq \ell_0 - 1$, we have
\[
    4n - 6\ell - \frac{53}{9} \geq 4n - 6(\ell_0 - 1) - \frac{53}{9} \geq 4n - 6\cdot\left( \frac{2n - 2}{3} - \frac{1}{3} \right) - \frac{53}{9} = \frac{1}{9} > 0.
\]
It follows that $S(G'_{\ell + 1}) - S(G'_\ell) > 0$. 
This shows that $\ell_0$ is the unique maximal point for $S(G'_\ell)$. This completes the proof of the theorem.
\end{proof}

\section{Appendix}
Note that Equations \eqref{eigen_main} and \eqref{e21} are quite similar. We have the following lemma on the existence of Laurent series of their solutions. Here we let $N=n-1$ for  Equation  \eqref{eigen_main} and $N=2n-4$ for Equation \eqref{e21}.
\begin{lemma}\label{lem:laurent}
Suppose that $\lambda$ satisfies the equation
\begin{equation} \label{eq:A1}
\lambda^2= N +\sum_{i=1}^\infty \frac{a_i}{\lambda^i},
\end{equation}
where $a_i=p_i \ell +q_i$'s are a linear function of $\ell$ for some constants $p_i$ and $q_i$ for each $i \in \mathbb{N}$.
Then $\lambda$ has the following Laurent series expansion in term of $\frac{1}{\sqrt{N}}$ near $\sqrt{N}$:
\begin{equation}
\lambda=\sqrt{N} +\sum_{i=1}^\infty c_i N^{-(i-1)/2},
\end{equation}
 and the following Laurent series expansion in term of $\frac{1}{\sqrt{N}}$ near $-\sqrt{N}$:
 \begin{equation}
\lambda=-\sqrt{N} +\sum_{i=1}^\infty (-1)^{i-1} c_i N^{-(i-1)/2},
\end{equation}
where the $c_i$'s are polynomials of $\frac{\ell}{N}$ of degree at most $i$.
\end{lemma}

\begin{proof}
Let $x=\frac{1}{\lambda}$ and $z^2=\frac{1}{N}$. Equation \eqref{eq:A1} can be rewritten as
\begin{equation} \label{eq:A2}
\frac{1}{z^2}= \frac{1}{x^2} - \sum_{i=1}^\infty a_ix^i.
\end{equation}
Thus
\begin{align*}
z &= \pm \left(\frac{1}{x^2} - \sum_{i=1}^\infty a_ix^i\right)^{-\frac{1}{2}} \\
&= \pm x \left(1 - \sum_{i=1}^\infty a_ix^{i+2}\right)^{-\frac{1}{2}} \\
&=\pm \frac{x}{\phi(x)}.
\end{align*}
Here $\phi(x)= \sqrt{1 - \sum_{i=1}^\infty a_ix^{i+2}}$. We have the following Taylor expansion of $\phi(x)$.
\begin{align}
  \phi(x) &=   \sqrt{1 - \sum_{i=1}^\infty a_ix^{i+2}} \\
  &=1+\sum_{k=1}^\infty \binom{\frac12}{k} \left(-\sum_{i=1}^\infty a_ix^{i+2}\right)^k   \label{eq:ai}\\
  &=1+ \sum_{j=3}^\infty b_j x^j, \label{eq:bi}
\end{align}
where $b_j$ are multi-variable polynomials of the $a_i$'s.

Consider the case $z>0$ first. Let $x=g(z)=\sum_{i=1}d_iz^i$ be the Taylor series of the inverse function in the equation
$z=\frac{x}{\phi(x)}$.
The Lagrange-B\"urmann formula states
$$d_n=[z^n]g(z)=\frac{1}{n}[x^{n-1}]\phi(x)^n.$$
Here $[x^r]$ is an operator which extracts the coefficient of $x^r$ in the Taylor series of a function of $x$. Note that $d_0=0$ and $d_1=1$.

Now taking the reciprocal of both sides, we have
\begin{align}
\lambda&=\frac{1}{x}\\
&=\frac{1}{z}\left(1+\sum_{i=1}^\infty d_{i+1}z^{i}\right)^{-1} \label{eq:di}\\
&=\frac{1}{z}\left(1+ \sum_{j=1}^\infty(-1)^j\left( \sum_{i=1}^\infty d_{i+1}z^i\right)^j\right)\\
&= \frac{1}{z}\left(1+  \sum_{j=1}^\infty f_j z^j\right) \label{eq:fj1} \\
&=\sqrt{N} +\sum_{i=1}^\infty f_i N^{-(i-1)/2},  \label{eq:fj}
\end{align}
where in \eqref{eq:fj}, the $f_i$s are multi-variable polynomials of the $d_i$s; in \eqref{eq:di} the $d_i$s are  multi-variable polynomials of the $b_i$s;
and the $b_i$s are multi-variable polynomials of the $a_i$s. Hence, the $f_i$s are multi-variable polynomials of the $a_i$s. 
Consider a typical monomial $a_{i_1}a_{i_2}\cdots a_{i_s}$ (with $a_1\leq a_2\leq \cdots \leq a_s$). 
By \eqref{eq:ai} and \eqref{eq:bi}, the monomial $a_{i_1}a_{i_2}\cdots a_{i_s}$ appears in $f_jz^j$ of \eqref{eq:fj1} only if
$$j\geq \sum_{t=1}^s (a_{i_t}+2) = 2s + \sum_{t=1}^s i_t$$
since $b_1=b_2=0$.
We have
\begin{equation}\label{eq:monomial-terms}
    a_{i_1}a_{i_2}\cdots a_{i_s} N^{-(j-1)/2}=
N^{-(j-1-2s)/2}
\prod_{t=1}^s \frac{a_{i_t}}{N}
=N^{-(j-1-2s)/2}\prod_{t=1}^s(p_{i_t} \frac{\ell}{N}+\frac{q_{i_t}}{N}).
\end{equation}

Expanding and grouping all items in $f_jN^{-(j-1)/2}$
in the Laurent series of $\lambda$, we can rewrite it as
\begin{equation}
    \lambda= \sqrt{N}+ \sum_{i=1}^\infty c_iN^{-(i-1)/2} 
\end{equation}
where $c_i$ is a polynomial of $\frac{\ell}{N}$. 
Note that in \eqref{eq:monomial-terms}, since $(j-1-2s)\geq \sum_{t=1}^s i_t-1\geq s-1$, a term $\left(\frac{\ell}{N}\right)^s$ can only appear
in $c_j$ for $j\geq s$. Thus, the degree of $c_i$ is at most $i$. 

By the symmetry between $z$ and $-z$ in \eqref{eq:A2}, we get the Laurent series expansion in term of $\frac{1}{\sqrt{N}}$ near $-\sqrt{N}$:
\begin{align*}
    \lambda&= \frac{1}{-z}\left(1+  \sum_{j=1}^\infty c_j (-z)^j\right)\\
    &=-\sqrt{N} +\sum_{i=1}^\infty (-1)^{i-1} c_i N^{-(i-1)/2}.
\end{align*}

\end{proof}
\end{document}